\documentclass[11pt,a4]{amsart}
\usepackage{amsfonts,amssymb}
\usepackage{mathabx}
\usepackage{tabulary}
\usepackage{esint}
\usepackage[svgnames]{xcolor} 
\usepackage[colorlinks,citecolor=red,pagebackref,hypertexnames=false,breaklinks]{hyperref}
\usepackage{pgf,tikz}
\usepackage{pdfsync}

\usepackage[bottom]{footmisc}
\usepackage{dsfont}
\usepackage{url}
\usepackage[utf8]{inputenc}
\usepackage[T1]{fontenc}
\usepackage{lmodern}
\usepackage{mathtools}  
\usepackage{lipsum}
\usepackage{mathrsfs}
\usepackage{color}

\usepackage{epsf,graphicx,epsfig,cite,latexsym}
\usepackage[export]{adjustbox}
\usepackage{mathtools}
\usepackage{latexsym, amsmath, amssymb, a4, epsfig}
\usepackage[T1]{fontenc}

\usepackage{listings}

\usepackage{algorithm}
\usepackage{algpseudocode}

 \usepackage{mathrsfs}
\newcommand{\no}[1]{#1}
\renewcommand{\no}[1]{}
\no{\usepackage{times}\usepackage[subscriptcorrection, slantedGreek, nofontinfo]{mtpro}
\renewcommand{\Delta}{\upDelta}}
\usepackage{color}

\numberwithin{algorithm}{section}
\numberwithin{figure}{section}

\newtheorem{lemma}{Lemma}[section]
\newtheorem{remark}{Remark}[section]

\newtheorem{proposition}{Proposition}[section]
\newtheorem{theorem}{Theorem}[section]

\date{\today}

\newcommand{\C}{\mathbb{C}}

\def\m{\mathcal}
\let\conjugatet\overline 
\def\cc{\conjugatet}
\def\w{\widetilde }
\def\ep{\varepsilon}

\def\t{\tau}
\def\b{\bigg}

\def\no{\nonumber}
\newcommand{\norm}[1]{\|#1\|}
\newcommand{\abs}[1]{\left|#1\right|}

\newcommand{\R}{\mathbb{R}}
\def\qed{\ifmmode\hbox{\hfill\sqb}\else{\ifhmode\unskip\fi%
\nobreak\hfil
\penalty50\hskip1em\null\nobreak\hfil\sqb
\parfillskip=0pt\finalhyphendemerits=0\endgraf}\fi}
\newcommand{\intt}{\int^t_0}

\newcommand{\ds}{\displaystyle}

\newcommand{\be}{\begin{equation}}
\newcommand{\ee}{\end{equation}}
\newcommand{\ba}{\begin{array}}
\newcommand{\ea}{\end{array}}
\newcommand{\bea}{\begin{eqnarray*}}
\newcommand{\eea}{\end{eqnarray*}}
\newcommand{\bean}{\begin{eqnarray}}
\newcommand{\eean}{\end{eqnarray}}

\def\sqw{\hbox{\rlap{\leavevmode\raise.3ex\hbox{$\sqcap$}}$%
\sqcup$}}
\def\sqb{\hbox{\hskip5pt\vrule width4pt height6pt depth1.5pt%
\hskip1pt}}

\begin{document}

\title{Recovering the initial condition in the One-Phase Stefan problem}

\author{Chifaa Ghanmi}
\address{Chifaa Ghanmi, Facult\'e des Sciences de Tunis, Universit\'e Tunis El Manar, Tunisia}
\email{Chifaa.Ghanmi@fst.utm.tn}

\author{Saloua Mani Aouadi}
\address{Saloua Mani Aouadi,  Facult\'e des Sciences de Tunis, Universit\'e Tunis El Manar, Tunisia}
\email{Saloua.Mani@fst.utm.tn}

\author{Faouzi Triki}

\address{Faouzi Triki,  Laboratoire Jean Kuntzmann,  UMR CNRS 5224, 
Universit\'e  Grenoble-Alpes, 700 Avenue Centrale,
38401 Saint-Martin-d'H\`eres, France}

\email{faouzi.triki@univ.grenoble-alpes.fr}

\thanks{
The work of Faouzi Triki is supported in part by the
 grant ANR-17-CE40-0029 of the French National Research Agency ANR (project MultiOnde).}

\begin{abstract}
We consider the problem of recovering the initial condition 
in the one-dimensional one-phase Stefan problem for the heat equation from the knowledge of the position of the
melting point. We first recall some properties of the free boundary solution.  Then we study the uniqueness and
 stability of the inversion. The principal contribution  of the paper is a new  logarithmic type stability  estimate 
 that shows that the inversion may be severely ill-posed. The proof is based on  integral equations representation
 techniques,  and  the unique continuation property for parabolic type solutions. We also present few  numerical examples   operating with  noisy synthetic  data.
   \end{abstract}

\maketitle

\section{Introduction}
Stefan problem is a  free boundary problem related to the heat  diffusion. It describes  the temperature spread  in a medium undergoing a phase change, for example water passing to ice. It can be found in many engineering settings where melting or freezing cause a boundary change \cite{gol2012inverse,rubinstein1971stefan}.\\

The Stefan problem for the heat equation consists in determining the temperature and the location of the melting front delimiting the different phases when the initial and boundary conditions are given.  It admits a unique solution which depends continuously on the data, assuming that the initial state and the source functions have the correct signs \cite{friedman1992free,ladyzhenskaia1968linear}. Conversely, the inverse Stefan problem is to recover initial condition from measurement of the moving boundary position. It is a non conventional Cauchy type  problem because of  the 
dependency of  free boundary on the initial and boundary conditions. \\

 In contrast with the direct Stefan problem only few theoretical results are available for the related inverse problem. Recently, the same authors studied the identification of a boundary flux condition  in the one dimensional 
 one-phase Stefan  problem \cite{ghanmi2020identification}. Most available published  materials have considered  the  numerical reconstruction of the initial condition or the  heat flux on the boundary \cite{jochum1980numerical,johansson2011method, wei2009reconstruction, hajiollow2020recovering}. Indeed, solving numerically  both direct and inverse Stefan problems could  be  difficult because of the free boundary, the nonlinearity and  the instability \cite{cannon1967stability, wrobel1983boundary, reemtsen1984method}.    \\

Our aim in this paper is to recover the initial state in the one-dimensional one-phase Stefan problem from the knowledge of the moving free boundary.  The problem of determining the initial condition  of linear parabolic equations has been the objective of a lot of papers in the last years  \cite{choulli2017logarithmic, li2009conditional, garcia2013heat, choulli2015various}.   On the other hand, to our knowledge, for parabolic free-boundary problems, this has not been considered in depth.  The uniqueness of the inversion has been established by several authors under different smoothness assumptions on the boundary influx and the initial condition \cite{kyner1959existence, cannon1967cauchy, cannon1967existence, friedman2008partial}. The  problem can also be related to the controllability of parabolic or hyperbolic  free-boundary problems \cite{fernandez2016controllability, fernandez2019local, ammari2017remark, geshkovski2019controllability, ammari2020weak}. \\

The outline of this paper is as follows. In section 2, we formulate the direct Stefan problem. We recall existence, uniqueness and stability results for the direct problem. In section 3, we set the inverse Stefan problem. The main stability estimate is provided in Theorem \ref{main theorem u0}. Finally,  we present  in section 4 numerical  examples for the  inverse Stefan problem utilizing noisy synthetic data.

\section{The direct problem}
In this section, we introduce the one dimension one-phase Stefan problem. Let $b>0$ and $T>0$ be two fixed constants. For any positive function $s\in C([0,T])$ satisfying $s(0)=b$, we define the open set $Q_{s,T}\subset \R_+ \times (0,T)$, by 
\begin{equation}\label{Qst}
 Q_{s,T}=\{ (x,t) |\ \  0<x<s(t), \quad 0<t<T \}.
\end{equation}
The direct Stefan problem consists in determining $u(x,t)\in C(\cc {Q_{s,T}})\cap C^{2,1} (Q_{s,T})$ and $s(t) \in C^1((0,T])\cap C([0,T])$,  satisfying 
 \begin{align}
    u_t-u_{xx}&=0,  &\quad in \quad  Q_{s,T}, \label{1.16} \\
   -u_x(0,t)&=h(t)>0, \quad &0<t<T, \label{1.17} \\
   u(x,0)&=u_0(x) \geq 0, \quad &0<x<b, \label{1.18}\\
   -u_x(s(t),t)&=\overset{.}{s}(t), \quad &0<t<T, \label{1.19}\\   
   u(s(t),t)&=0, \quad &0<t<T, \label{1.20}\\
   s(0)&=b. \label{1.21}
  \end{align}
where $h\in C([0,T])$ and $u_0\in C([0,b])$ are given. \\

Let $H>0$ be a fixed constant. Throughout the paper we assume 
\begin{equation}\label{1.22}
 h\in C([0,T]), \ \ h(t)>0, \ \ 0\leq t\leq T,
\end{equation}
\begin{equation}\label{1.23}
 u_0\in C([0,b]), \ \ 0\leq u_0(x)\leq H(b-x), \ \ 0\leq x\leq b.
\end{equation}
Let 
\begin{equation}\label{M=}
\w M=\max \{ \norm{h}_\infty , H\}. 
\end{equation}
Next, we give a result of existence and uniqueness of the Stefan problem (\ref{1.16})-(\ref{1.21}).
\begin{theorem}\label{theorem2.1}
The problem (\ref{1.16})-(\ref{1.21}) admits a unique solution $(s(t),u(x,t))$. In addition the solution $(s(t),u(x,t))$ satisfies
\begin{align}
0&<u(x,t) \leq \w M(s(t)-x),  & \text{in} \ \  Q_{s,T} \label{1.28} \\
0&\leq -u_x(s(t),t)=\dot{ s} (t) \leq \w M ,  & \text{in } (0,T]. 
\end{align} 
\end{theorem}
The proof of this theorem is based on the maximum principle for parabolic type equations, and the fixed point Theorem. The existence and uniqueness of the Stefan problem has been studied by various authors under different smoothness assumptions (see for instance \cite{cannon1971remarks,cannon1967stability} and references therein).\\

Next we give some useful properties of the solutions of the Stefan problem as well as their stability with respect to boundary and initial data. \\

Consider two sets $(h_i(t),u_{0,i}(x)),\ i=1,2$, of Stefan data satisfying assumptions (\ref{1.22}) and (\ref{1.23}). By the Theorem \ref{theorem2.1}, each one of these two problems admits a unique solution $(s_i,u_i)$. 
\begin{theorem}\label{stability of direct problem}
If $b_1<b_2$, then the free boundaries $s_i(t), i= 1, 2$, corresponding to the data $(h_i(t),u_{0,i}(x)),\ i=1, 2$, 
satisfy 
\begin{equation}\label{directstability}
\abs{s_1(t)-s_2(t)}\leq C\bigg( b_2-b_1+\int^{b_1}_0 \abs{u_{0,1}(x)-u_{0,2}(x)}dx +\int^{b_1}_{b_2} u_{0,2}(x) dx +\intt \abs{h_1(\t)-h_2(\t)}d\t  \bigg),
\end{equation}
for $0\leq t\leq T,$ where $C>0$ only depends on $T$ and $\w M$.
\end{theorem}
The proof of this theorem is detailed in \cite{cannon1967stability}. The constant $C>0$ appearing in Theorem \ref{stability of direct problem} is in fact exponentially increasing as function of $T$ and $\w M$.

\section{The Inverse Problem}
\setcounter{equation}{0}

In this section,  we consider the problem of determining the initial condition  $u(x,0)=u_0(x)$ in the system below from the knowledge of the moving boundary $s\in C^1([0,T])$, and the heat flux $h\in C([0,T])$.
\begin{eqnarray}\label{Inverse problem u0}
(\mathcal{P})\left\{ \begin{array}{rll}
   u_t-u_{xx}&=0,  &0<x<s(t), \ 0<t<T ,\cr
   -u_x(0,t)&=h(t) \geq 0,  &0<x<b,\cr
   -u_x(s(t),t)&=\overset{.}{s}(t), &0<t<T,\cr   
   u(s(t),t)&=0, &0<t<T,\cr
   s(0)&=b. & \cr 
\end{array}\right.
\end{eqnarray}
We next present the main result of the paper.
\begin{theorem}\label{main theorem u0}
Let $\hat{H},\widehat{M}\geq 1$ be two fixed constants. Let $h\in C([0,T])$ be a given strictly positive function satisfying $\norm{h}_\infty\leq \hat{H}$, $u_0,\widetilde u_0 \in H^1(0,b)$,  verifying
\begin{align}
0&\leq u_0(x)\leq H(b-x) ,\ 0\leq x\leq b,  \\
0&\leq \widetilde u_0(x)\leq H(b-x),\ 0\leq x\leq b,  \\
& \norm{u_0}_{H^1(0,b)},\norm{\widetilde u_0}_{H^1(0,b)} \leq \widehat{M} .
\end{align}
Let $u$ and $\w u$ be the solutions to the system (\ref{Inverse problem u0}) associated respectively to  $u_0$ and $\w u_0$. Denote $s$ and $\w s$ the free boundaries of respectively the solutions $u$ and $\w u$. Then  there 
exists a constant $\varepsilon_0 = \varepsilon(\widehat{M}, b, T,\hat{H}, u_0)>0$  such that  if $\norm{s-\w s}_{W^{1,\infty}}< \varepsilon_0$,
 the following  estimate 
\begin{equation} \label{mainineq}
\norm{u_0 -\w u_0}_{L^2(0,b)} \leq \frac{C}{\left| \ln( \left| \ln \left(\norm{s-\w s}_{W^{1,\infty}} \right) \right| )\right| ^{\frac{1}{4}} },
\end{equation}
holds, where the constant $C= C(\widehat{M},b, T, \hat{H}, u_0)>0$.
\end{theorem}

The double logarithmic  stability estimate \eqref{mainineq} shows that the inverse  Stefan problem may be severely  ill-posed. 
The obtained result is in agreement with  known stability estimates for standard  Cauchy problems of linear  parabolic equations (see for instance Theorem 1.1 in \cite{choulli2017logarithmic}). Therefore  the moving  boundary causing the nonlinearity and  higher complexity of  the direct problem,  does not seem to modify the nature of the Cauchy inversion.\\

The rest  of this section is devoted to the proof of the main Theorem \ref{main theorem u0}. \\

Now, we recall an integral representation of  Stefan problem's solution which can be found in \cite{ghanmi2020identification, cannon1967cauchy,friedman1964partial}.
\begin{lemma} 
Let $(u,s)$ be a solution of the direct Stefan problem (\ref{1.16})-(\ref{1.21}) verifying the assumptions of Theorem \ref{main theorem u0}. Then the solution $u$ satisfies
\begin{equation}\label{u}
u(x,t)= \intt N(x,0;t,\t)h(\t) d\t -\intt N(x,s(\t);t,\t) \dot s(\t) d\t+ \int^b_0 N(x,\xi;t,0)u_0(\xi) d\xi,
\end{equation}
where the Neumann function $N$ is defined by
\begin{equation}
\label{N}
N(x,\xi;t,\tau)=K(x,\xi;t,\tau)+ K(-x,\xi;t,\tau), 
\end{equation}
with
\begin{equation}
K(x,\xi;t,\tau)=\frac{1}{2\sqrt{\pi(t-\tau)}} \exp\bigg(\frac{-(x-\xi)^2}{4(t-\tau)} \bigg), \ \ \tau<t.
\end{equation} 
\end{lemma} 

Taking the integral representation  \eqref{u} of $u$  and $\w u$  at  respectively $x= s(t)$,
and $x= \tilde s(t)$, we find
\begin{align*}
u(s(t),t)=& \intt N(s(t),0;t,\t)h(\t)d\t -\intt N(s(t),s(\t);t,\t) \dot s(\t) d\t+ \int^b_0 N(s(t),\xi;t,0)u_0(\xi) d\xi, \\
\w u(\w s(t),t)=& \intt N(\w s(t) ,0;t,\t) h(\t) d\t -\intt N(\w s(t),\w s(\t);t,\t) \dot {\w s}(\t) d\t + \int^b_0 N(\w s(t),\xi;t,0)\w u_0(\xi)d\xi,
\end{align*}
for all $ t \in (0,T)$. \\

Consequently
\begin{align}
u(s(t),t)-\w u(\w s(t),t)& \hspace{-1mm}= \hspace{-2mm}\int_0^b N(\w s(t),\xi;t,0)[u_0(\xi)- \w u_0(\xi)]d\xi +  \intt [N(s(t),0;t,\t)- N(\w s(t),0;t,\t)] h(\t) d\t \no \\ 
& \hspace{-2cm}+ \intt [N(\w s(t),\w s(\t);t,\t)- N(s(t),s(\t);t,\t)] \dot {\w s}(\t) d\t + \intt N(s(t),s(\t);t,\t)[\dot {\w s}(\t)-\dot s(\t)] d\t \no \\ 
& \hspace{-2cm}+ \int^b_0 [N(s(t),\xi;t,0)-N(\w s(t),\xi;t,0)] u_0(\xi) d\xi \no \\
& = \sum_{i=1}^5 I_i  =0, \quad \text{for all } t \in (0,T). \label{sumofI} 
\end{align}

Our aim now is to estimate each of the integrals $I_i, \ i=2,...,5$ in terms of the difference between $s$ and $\w s$.
\begin{lemma}\label{lll}
Let $h\in C([0,T])$ be a given positive function, $u_0, \w u_0\in H^1(0,b)$ satisfying $\norm{u_0}_{H^1(0,b)},$\\ 
$\norm{\w u_0}_{H^1(0,b)} \leq \widehat{M} $. 
Denote $s$ and $\w s$ the free boundaries of respectively the solution $u$ and $\w u$. We have the following inequality  
\bean
\left\vert \int_0^b  N(x,\xi; t,0)[u_0(\xi)- \w u_0(\xi)]d\xi \right\vert \leq C \norm{s-\w s}_{W^{1,\infty}}, \quad x\geq \tilde s(t), \quad 0\leq t\leq T,
\eean
where the constant $C>0$ only depends on $u_0,b,\hat{H},\widehat{M}$ and $T$. 
\end{lemma}

\begin{proof}
In the following proof $C$ that  only depends on $u_0,b,\hat{H},\widehat{M}$ and $T$, stands for a generic constant strictly larger than zero.

Let $\ep$ be a fixed constant satisfying $0<\ep<t<T$, and set 
\begin{equation*}
I_2^\ep = \int_{0}^{t-\ep } [N(s(t),0;t,\t)- N(\w s(t),0;t,\t)] h(\t) d\t.
\end{equation*} 
From the mean value theorem, we deduce 
\begin{equation}\label{estimationexp}
\abs{e^{-a} - e^{-b}} \leq e^{-\min(a,b)} \abs{a-b}, \quad a,b\geq 0. 
\end{equation}
Recalling that $s(t),\w s(t)>b$ for all $t\in [0,T]$, and applying the inequality (\ref{estimationexp}) to $I_2^\ep$, we obtain 
\begin{align}
\abs{I_2^\ep} &= \abs{ \int_{0}^{t-\ep } \frac{1}{2\sqrt{\pi(t-\tau)}}\b[ 2\exp\bigg (\frac{-s(t)^2}{4(t-\tau)} \bigg)  -2\exp\bigg (\frac{-\w s(t)^2}{4(t-\tau)} \bigg)\b]  h(\t)  d\tau}, \no \\
&\leq \int_{0}^{t-\ep } \frac{e^{\frac{-b^2}{4(t-\tau)}}}{\sqrt{\pi(t-\tau)}} \abs{\dfrac{s^2(t)- \tilde{s}^2(t)}{4(t-\tau)}} \abs{h(\tau)} d\tau, \no \\
&\leq \int_{0}^{t-\ep } \dfrac{e^{\frac{-b^2}{4(t-\tau)}}}{2\sqrt{\pi}(t-\tau)^{3/2}} d\tau \max{ \big( \norm{s}_\infty,\norm{\tilde{s}}_\infty\big)} \norm{h}_\infty  \norm{s-\tilde{s}}_\infty,    \no \\
&\leq   \int^T_0 \dfrac{e^{\frac{-b^2}{4r}}}{2\sqrt{\pi r^{3/2}}} dr  \max{ \big( \norm{s}_\infty,\norm{\tilde{s}}_\infty\big)} \norm{h}_\infty  \norm{s-\tilde{s}}_\infty  \label{I2ep}.
\end{align}
Then 
\begin{equation}\label{I2}
\abs{I_2}\leq C\norm{s-\w s}_\infty.
\end{equation}
Now, let 
\begin{equation*}
I_3^\ep= \int_0^{t-\ep} [N(\w s(t),\w s(\t);t,\t)- N(s(t),s(\t);t,\t)] \dot {\w s}(\t) d\t, 
\end{equation*}
which is equivalent to 
\begin{align*}
I_3^{\ep}&= \int_{0}^{t-\ep } \frac{1}{2\sqrt{\pi(t-\tau)}}\b[ \exp\bigg (\frac{-(\w s(t)-\w s(\t))^2}{4(t-\tau)} \bigg)\nonumber 
-\exp \b( \frac{-(s(t)- s(\t))^2}{4(t-\tau)} \b)  \\
&+ \exp\bigg(\frac{-(\w s(t)+\w s(\t))^2}{4(t-\tau)}\b) 
 - \exp\bigg(\frac{-( s(t)+ s(\t))^2}{4(t-\tau)} \b) \b]  \dot {\w s}(\t) d\t.  \nonumber
\end{align*}
For $0 <\tau < t-\ep $, using (\ref{estimationexp}), we obtain
\begin{align*}
&\abs{\exp\bigg (\frac{-(\w s(t)-\w s(\t))^2}{4(t-\tau)} \bigg)- \exp\bigg(\frac{-( s(t)- s(\t))^2}{4(t-\tau)}\b)}
 \leq \frac{1}{4(t-\t)} \abs{(\w s(t)-\w s(\t))^2 -( s(t)- s(\t))^2} \\
& \leq  \frac{1}{4(t-\t)} \abs{(\w s(t)-\w s(\t))-(s(t)-s(\t))} \abs{(\w s(t)-\w s(\t))+(s(t)-s(\t))}. 
\end{align*}
Since $s,\w s$ are Lipschitz functions, we have 
\begin{equation*}
\dfrac{\abs{s(t)-s(\tau)}}{t-\tau},\dfrac{\abs{\w s(t)-\w s(\tau)}}{t-\tau} \leq \max(\norm{s}_{C^{0,1}([0,T])},\norm{\w s}_{C^{0,1}([0,T])}) \leq C \max (\norm{s}_{W^{1,\infty}}, \norm{\w s}_{W^{1,\infty}}).
\end{equation*}  
Hence 
\begin{equation*}
\abs{\exp\bigg (\frac{-(\w s(t)-\w s(\t))^2}{4(t-\tau)} \bigg)- \exp\bigg(\frac{-( s(t)- s(\t))^2}{4(t-\tau)}\bigg)}
\leq C \max (\norm{s}_{W^{1,\infty}}, \norm{\w s}_{W^{1,\infty}}) \norm{s-\w s}_\infty. 
\end{equation*}
Consequently
\begin{equation*}
\abs{I_3^\ep} \leq C \int_{0}^{t-\ep } \frac{1}{2\sqrt{\pi(t-\tau)}} d\tau \max (\norm{s}_{W^{1,\infty}}, \norm{\w s}_{W^{1,\infty}}) \norm{\w s}_{W^{1,\infty}} \norm{s-\w s}_\infty  \leq C \norm{s-\w s}_\infty, 
\end{equation*}
which yields 
\begin{equation}\label{I3}
\abs{I_3} \leq C \norm{s-\w s}_\infty.
\end{equation}
For the equation $I_4$, we have 
\begin{align}
|I_4|&= \b|  \intt N(s(t),s(\t);t,\t)[\dot {\w s}(\t)-\dot s(\t)] d\t\b| \no \\
&=\b|\intt \frac{1}{2\sqrt{\pi(t-\tau)}}\b[ \exp\bigg (\frac{-(s(t)-s(\t))^2}{4(t-\tau)} \bigg) + \exp\bigg(\frac{-(s(t)+s(\t))^2}{4(t-\tau)} \bigg)\b]\b(\dot{\w s}(\t)-\dot s(\t) \b) d\t \b| \no \\
&\leq  \intt \frac{1}{\sqrt{\pi(t-\tau)}}  \abs{\dot{\w s}(\t)- \dot s(\t)} d\tau \leq \intt \frac{1}{\sqrt{\pi(t-\tau)}} d\tau \norm {\dot{\w s}- \dot s}_\infty \leq \frac{2}{\sqrt{\pi}} \sqrt{T} \norm{ \dot{\w s}- \dot s}_\infty  \no \\
&\leq C \norm{s-\w s}_{W^{1,\infty}}. \label{I4}
\end{align}
Next, we estimate $I_5$.
\begin{align*}
|I_5|&= \b |\int^b_0 [N(s(t),\xi;t,0)-N(\w s(t),\xi;t,0)] u_0(\xi) d\xi \b| \no \\
&=\b | \int^b_0 \frac{1}{2\sqrt{\pi t}}\b[ \exp\bigg (\frac{-(s(t)-\xi)^2}{4t} \bigg) + \exp\bigg(\frac{-(s(t)+\xi)^2}{4t} \b)\\
&- \exp\bigg (\frac{-(\w s(t)-\xi)^2}{4t} \bigg) \no 
-\exp\bigg(\frac{-(\w s(t)+\xi)^2}{4t} \b)\b]u_0(\xi) d\xi \b| \no \\
&\leq \int^b_0 \frac{1}{2\sqrt{\pi t}} \abs{\exp\bigg (\frac{-(s(t)-\xi)^2}{4t} \bigg)-\exp\bigg (\frac{-(\w s(t)-\xi)^2}{4t} \bigg)} \abs{u_0(\xi)} d\xi \\
&+\int^b_0 \frac{1}{2\sqrt{\pi t}} \abs{\exp\bigg(\frac{-(s(t)+\xi)^2}{4t} \b)- \exp\bigg(\frac{-(\w s(t)+\xi)^2}{4t} \bigg) }  \abs{u_0(\xi)} d\xi. 
\end{align*}
Since $s(t)>0$, applying again the inequality (\ref{estimationexp}) to the previous estimate of $I_5$, leads to 
\begin{align*}
\abs{I_5} &\leq \frac{1}{8\sqrt{\pi}} \int^b_0 \big( \abs{(s(t)+\xi)^2-(\w s(t)+\xi)^2} +\abs{(s(t)-\xi)^2-(\w s(t)-\xi)^2}\big)
\dfrac{e^{\frac{-\xi^2}{t}}}{t^{3/2}} d\xi \abs{u_0(\xi)} d\xi \\
 &\leq \frac{1}{8\sqrt{\pi}} \int^b_0 \abs{s(t)-\w s(t)} \big(\abs{s(t)+\w s(t)+2\xi} +\abs{s(t)+\w s(t)-2\xi}\dfrac{e^{\frac{-\xi^2}{t}}}{t^{3/2}} d\xi \norm{u_0}_\infty \\
 &\leq \frac{1}{2\sqrt{\pi}} \int^b_0 \dfrac{e^{\frac{-\xi^2}{t}}}{\sqrt{t}} d\xi \dfrac{\abs{s(t)-\w s(t)}}{t} \max{(\norm{s}_\infty, \norm{\w s}_\infty}) \norm{u_0}_\infty \\
&\leq \dfrac{1}{4} \dfrac{\abs{s(t)-\w s(t)}}{t} \max{(\norm{s}_\infty, \norm{\w s}_\infty}) \norm{u_0}_\infty.
\end{align*}
Since $s(0)=\w s(0)=b$, we have 
\begin{equation*}
\dfrac{\abs{s(t)-\w s(t)}}{t} \leq 2\norm{s-\w s}_{C^{0,1}([0,T])} \leq C \norm{s-\w s}_{W^{1,\infty}}.
\end{equation*}
Therefore 
\begin{equation}\label{I5}
\abs{I_5} \leq C \norm{s-\w s}_{W^{1,\infty}}.
\end{equation}
By combining (\ref{I2}), (\ref{I3}), (\ref{I4}) and (\ref{I5}), we obtain 
\begin{equation}\label{sum}
\sum_{i=2}^5 I_i \leq C \norm{s-\w s}_{W^{1,\infty}}.
\end{equation}
We deduce from (\ref{sumofI}) and (\ref{sum}) that 
\begin{equation}
\abs{I_1} \leq C \norm{s-\w s}_{W^{1,\infty}}.
\end{equation}

By the maximum principle for the heat equation \cite{friedman2008partial}, we have 
\begin{equation}
\abs{I_1(x)}=\abs{\int_0^b N(x,\xi;t,0)[u_0(\xi)- \w u_0(\xi)]d\xi} \leq C \norm{s-\w s}_{W^{1,\infty}}, \quad x\geq  \w s(t),  \quad 0\leq t\leq T,
\end{equation}
which ends the proof of the lemma.
\end{proof}

Now, we fix $c\in \mathbb R$, and let $\Pi_c^+ $ be the  right half-plane
\begin{equation}
{\Pi_c^+ = \Big\{ z= s+i \tau: \tau \in \mathbb R, \, s>c  \Big\}.}
\end{equation}
 We consider the Hardy space $H^2 (\Pi_c^+)$ {\cite{rudin2006real}}, defined as the space of holomorphic functions $h$ in $\Pi_c^+$ for which
\[
\sup_{s>c} \int_{-\infty}^{+\infty} |h(s+it)|^2 dt <\infty.
\]
Let 
\[
 L^2_c(\mathbb R_+)= \Big\{ e^{-ct}f(t) \in L^{2}(\mathbb R_+)  \Big\}.
 \]
We define the Laplace transform of a function $f$ by 
\begin{equation}
\m L f(z)= \int_0^{+\infty} e^{-z \tau }f(\tau)d\tau.
\end{equation} 
 
\begin{lemma}\label{inversionlemma}
The Laplace transform $\mathcal L: \; L^2_c(\mathbb R_+) \rightarrow H^2 (\Pi_c^+)$ is an invertible bounded operator. 
 In addition for $s\geq c $, we have 
 \bean \label{identity}
\frac{1}{2\pi} \int_{-\infty}^{+\infty} |\mathcal L f(s+it)|^2 dt  = \int_0^{+\infty} e^{-2s\tau}|f(\tau)|^2 d\tau.
 \eean
\end{lemma}
\begin{proof}
For $f \in  L^2_c(\mathbb R_+)$, 
the Laplace transform 
\bean \label{Laplace}
\mathcal L f(z)= \int_0^{+\infty} e^{-z \tau} f(\tau) d\tau,
\eean
has an holomorphic extension  to the right half-plane $\Pi_c^+$. Moreover one can easily check 
that 
\[
\sup_{s>c} \int_{-\infty}^{+\infty} | \mathcal L f (s+it)|^2 dt <\infty.
\]
We deduce from \eqref{Laplace} 
\bea
\mathcal L f (s+it) =  \int_0^{+\infty} e^{-i t \tau} e^{-s\tau} f(\tau) d\tau, \quad t\in \mathbb R.
\eea
We remark that $t\rightarrow \mathcal L f (s+it)$ is actually the  Fourier transform of $e^{-s\tau} f(\tau) \chi_{(0, +\infty)}(\tau)$
which lies in  $L^2(0, +\infty)$ for $s=c$ and in $L^2(0, +\infty)\cap L^1(0, +\infty)$ for $s>c$. Applying the 
classical  inverse Fourier transform, we get 
\bea
e^{-s\tau} f(\tau)  = \frac{1}{2\pi} \int_{-\infty}^{+\infty} e^{i\tau t} \mathcal L f(s+it) dt, \quad t>0.
\eea
The identity  \eqref{identity} is then a direct consequence of the Parseval identity.\\

Finally, Paley-Wiener Theorem  \cite{rudin2006real} shows that the  Laplace transform is surjective from $L^2_c(\mathbb R_+)$  onto  $H^2 (\Pi_c^+)$.
  \end{proof}  
  
Now, we give the proof of the main Theorem \ref{main theorem u0}. Further $C>0$ denotes
a generic  constant that depends on $\widehat{M},b, T,\hat{H}, u_0$.
\begin{proof}
We first  estimate $I_1$. We have 
\begin{align*}
I_1(x)&= \int_0^b N(x,\xi;t,0)(u_0(\xi)- \w u_0(\xi))d\xi \\
&= \dfrac{1}{2\sqrt{\pi t}} \int_0^b \big[ \exp\big (\frac{-(x-\xi)^2}{4t} \big) +\exp\big (\frac{-(x+\xi)^2}{4t} \big) \big] (u_0(\xi)- \w u_0(\xi))d\xi\\
&= \dfrac{1}{2\sqrt{\pi t}} \int_{-b}^b  \exp\big (\frac{-(x+\xi)^2}{4t} \big)  v_0(\xi) d\xi,
\end{align*}
where 
\begin{equation*}v_0(\xi)=
\left\lbrace 
\begin{array}{l}
 u_0(\xi)-\w u_0(\xi) \quad  \text{if } 0<\xi<b, \\
 u_0(-\xi)-\w u_0(-\xi) \quad \text{if } -b<\xi<0. 
 \end{array}
\right.
\end{equation*}

For fixed $t>0$, applying the changes of variables  $y=\frac{x-\w s(t)}{2t}$, and $\xi= r -b$, we  obtain
\begin{align*}
I_1( 2t y+\w s(t))&= \dfrac{1}{2\sqrt{\pi t}} \exp(-ty^2-(\w s(t)-b) y)  \int_{0}^{2b} e^{-y r} f(r) dr,
\end{align*}
with  \[f(r) =  \exp(-\frac{(\w s(t) -b+r)^2}{4t})v_0(r-b).\] 

Therefore 
 \bean \label{eee}
 \mathcal{L}(f)(y)=\displaystyle \int_0^{2b}  e^{-yr} f(r)dr = \displaystyle 2\sqrt{\pi t} \exp(ty^2+(\w s(t)-b) y) I_1(2ty+\w s(t)), \; \;  \textrm{for  }  y>0.\eean
Since $f \in L^\infty(0,+\infty)$, we have 
\begin{equation}
\norm{f}_{L^\infty(0,+\infty)} \leq \underset{0\leq r \leq 2b}{\sup}  \exp(-\frac{(\w s(t) -b+r)^2}{4t}) \abs{v_0(r-b)}
 \leq 2 \norm{u_0}_{L^\infty(0,b)}\leq 2bH.
\end{equation}
Let $\Omega$ be the right half plane, that is $\Omega=\{ z\in \C, \ Re(z)>0, \ Im(z) \in \R \}$. \\

For $z\in \Omega$, define 
\begin{equation}
F(z)=\mathcal{L}(f)(y+ip)= \int_0^{2b}  e^{-(y+ip)r} f(r)dr.
\end{equation}
We remark that the function $z\rightarrow F(z)$ is holomorphic in $\Omega$, and we have 
\begin{align*}
\abs{F(z)} \leq 2b  \norm{f}_{L^\infty(0,+\infty)} \leq 4b^2 H .
\end{align*}
Consequently
\begin{equation}
\abs{F(z)} \leq M, \quad z \in \Omega,
\end{equation}
where $M=\max(1,4b^2H)$.

\begin{proposition}
Denote $w_\pm(z)=\frac{2}{\pi} (\frac{\pi}{2}\mp arg(z))$, the harmonic measure of 
$]0,+\infty[ \times \R_\pm$. It is the unique solution to the system
\begin{align}
\Delta w_\pm(z) &= 0, \quad z\in ]0,+\infty[ \times \R_\pm, \\
w_\pm(z) &=1, \quad \Im(z)=0, \\
w_\pm(z) &=0, \quad { \Re(z)=0}.
\end{align}
\end{proposition}

The holomorphic unique continuation of the function $z\rightarrow F(z)$ using the two constants Theorem \cite{nevanlinna1970analytic}, gives 
\begin{equation}
\abs{F(z)} \leq M^{1-w_\pm(z)} \abs{F(z)}^{w_\pm(z)}_{L^\infty(\{0\} \times \R_\pm)} \leq M \ep^{w_\pm(z)}, 
\end{equation}
where $\ep=\abs{F(z)}_{L^\infty(\{0\} \times \R_+)}$. \\

We further consider two different situations. Let $B$ a strictly positive constant. If $\varepsilon \geq 1$, we obtain 
\begin{align}
\int_0^B\abs{F(1+ip)}^2 dp &\leq  \int_0^B\abs{M \ep^{w_+(1+ip)} }^2 dp \leq B M^2 \ep^{2\underset{0<p<B}{\max} w_+(1+ip) } \no \\
&\leq B M^2  \ep^{2(1-\frac{2}{\pi} arctan(0))}\leq  B M^2  \ep^2 \label{FBmax}.
\end{align}
We have by analogy
\begin{align}
\int^0_{-B}\abs{F(1+ip)}^2 dp &\leq  \int^0_{-B} \abs{M \ep^{w_-(1+ip)} }^2 dp \leq B M^2 \ep^{2\underset{-B<p<0}{\max} 
w_-(1+ip) } \no \\
&\leq B M^2  \ep^{2(1+\frac{2}{\pi} arctan(0))}  \leq B M^2  \ep^2      \label{F-Bmax}.
\end{align}
We deduce from  inequalities (\ref{FBmax}) and (\ref{F-Bmax}) that 
\begin{equation}
\int_{-B}^B \abs{F(1+ip)}^2 dp \leq 2 B M^2  \ep^2.\label{*max}
\end{equation}

Since $\norm{u_0}_{H^1(0,b)}\leq \widehat{M}$, then 
\begin{equation}
\norm{f}_{H^1(0,2b)} \leq C \widehat{M}.
\end{equation}
Recall that $f(0)=f(2b)=0$. Therefore, we have 
\begin{align}
\abs{p} \abs{F(1+ip)}&= \abs{\int_0^{2b} p e^{-ipr} f(r)e^{-r}dr } = \abs{\int_0^{2b} e^{-ipr}(f(r)e^{-r})'dr } \no \\
&\leq C \widehat{M} \label{**}.
\end{align}
Using  inequalities (\ref{*max}) and (\ref{**}) gives
\begin{align}
\int_\R \abs{F(1+ip)}^2 dp &= \int^B_{-B} \abs{F(1+ip)}^2 dp+\int_{\abs{p}>B} \abs{F(1+ip)}^2 dp \no \\
&\leq 2 B M^2  \ep^2 +C\widehat{M}\dfrac{1}{B} \leq C_1(B \ep^2 +\dfrac{1}{B}),
\end{align}
where $C_1= 2M^2+C\widehat{M}>0$.

In this case by taking $B=1$, we have 
\begin{equation}\label{eqqq1}
\int_\R \abs{F(1+ip)}^2 dp \leq 2C_1 \ep^2. 
\end{equation}
If $0<\varepsilon < 1$, we have 
\begin{align}
\int_0^B\abs{F(1+ip)}^2 dp &\leq  \int_0^B\abs{M \ep^{w_+(1+ip)} }^2 dp \leq B M^2 \ep^{2\underset{0<p<B}{\min} w_+(1+ip) } \no \\
&\leq B M^2  \ep^{2(1-\frac{2}{\pi} arctan(B))} \label{FB}.
\end{align}
We have by analogy
\begin{align}
\int^0_{-B}\abs{F(1+ip)}^2 dp &\leq  \int^0_{-B} \abs{M \ep^{w_-(1+ip)} }^2 dp \leq B M^2 \ep^{2\underset{0<p<B}{\min} 
w_-(1+ip) } \no \\
&\leq B M^2  \ep^{2(1-\frac{2}{\pi} arctan(B))}\label{F-B}.
\end{align}
We deduce from  inequalities (\ref{FB}) and (\ref{F-B}) that 
\begin{equation}
\int_{-B}^B \abs{F(1+ip)}^2 dp \leq 2 B M^2  \ep^{2(1-\frac{2}{\pi} arctan(B))}.\label{*}
\end{equation}
Using  inequalities (\ref{*}) and (\ref{**}) gives
\begin{align}
\int_\R \abs{F(1+ip)}^2 dp &= \int^B_{-B} \abs{F(1+ip)}^2 dp+\int_{\abs{p}>B} \abs{F(1+ip)}^2 dp \no \\
&\leq 2 B M^2  \ep^{2(1-\frac{2}{\pi} arctan(B))} +C\widehat{M}\frac{1}{B} \no \\
&\leq C_1 ( B \ep^{2(1-\frac{2}{\pi} arctan(B))} +\frac{1}{B}), \label{c1}
\end{align}
where $C_1= 2M^2+C\widehat{M}>0$. 

We fix $B\geq 1$. On one hand  we have $1-\dfrac{1}{3B^2}\geq \dfrac{1}{2}$, then 
\begin{equation*}
\frac{1}{B}(1-\frac{1}{3B^2})\geq \frac{1}{2B}.
\end{equation*}
On the other hand, we have 
\begin{align*}
\frac{\pi}{2}-arctan(B)&=arctan(\frac{1}{B}) \\
&\geq \frac{1}{B} -\frac{1}{3B^3} \\
&\geq \frac{1}{B} (1-\frac{1}{3B^2}).
\end{align*}
Then, for $B\geq 1$, we have 
\begin{equation}\label{c11}
\frac{\pi}{2} -arctan(B) \geq \frac{1}{2B}.
\end{equation}
Combining  inequality (\ref{c11}) with  estimate (\ref{c1}) yields 
\begin{equation}
\int_\R \abs{F(1+ip)}^2 dp \leq C_1 ( B \ep^{\frac{2}{\pi B}} +\frac{1}{B}) , B \geq 1.
\end{equation}
Now, by taking $B \ep^{\frac{2}{\pi B}}=\frac{1}{B}$, we obtain 
\begin{equation}\label{in1}
\int_\R \abs{F(1+ip)}^2 dp \leq \frac{C}{B} , \quad B \geq 1, 
\end{equation}
and 
\begin{equation}
-\frac{1}{\pi}\ln(\ep)=B \ln(B).
\end{equation}
We remark that $\dfrac{\ln(B)}{B}\leq e^{-1}$, for $B>e$. \\

Therefore we have $-\dfrac{1}{\pi}\ln(\ep)=B \ln(B)\leq B^2e^{-1}$.
In other words 
\begin{equation}\label{in2}
\frac{1}{B} \leq \sqrt{\frac{\pi}{e}} \frac{1}{\sqrt{\abs{\ln(\ep)}}}, \quad B>e.
\end{equation}
Substituting (\ref{in2}) in the estimation (\ref{in1}), we obtain 
\begin{equation} \label{eqqq2}
\int_\R \abs{F(1+ip)}^2 dp \leq \frac{C}{\sqrt{\abs{\ln(\ep)}}}.
\end{equation}
We deduce from inequalities \eqref{eqqq1} and \eqref{eqqq2} that
\begin{equation}
\int_\R \abs{F(1+ip)}^2 dp \leq \phi^2(\ep),
\end{equation}
where
\bean \label{phi}
\phi(\ep)  = \left\{ \ba{llcc}
C \ep, &\textrm{if  } \ep \geq 1,\\
\frac{C}{\abs{\ln(\ep)}^{\frac{1}{4}} } &\textrm{if  } \ep < 1.
\ea
\right.
\eean
By Fourier Plancherel in Lemma \ref{inversionlemma},  we have 
\begin{equation}
\int_0^{2b} \abs{f(r)}^2 e^{-2r}dr= \dfrac{1}{2\pi}\int_\R \abs{F(1+ip)}^2 dp  \leq 
\phi^2(\ep).
\end{equation}
Hence 
\begin{equation}
\norm{f}^2_{L^2(0,2b)} \leq \phi^2(\ep).
\end{equation}
Finally, we obtain 
\bean \label{bbb}
\norm{u_0-\w u_0}_{L^2(0,b)} \leq \phi(\|F(z)\|_{L^\infty(\{0\} \times \R_+)}).
\eean
On the other hand,  equation \eqref{eee} leads to  
\[
F(z) =  2\sqrt{\pi t}\exp(tz^2+(\w s(t)-b) z) I_1(2tz+\w s(t))
\]
for $z \in \mathbb R_+$. \\

Since $F(y)$ is continuous on $\mathbb R_+$,  and satisfies $\underset{y \rightarrow +\infty}\lim F(y) = 0$, it reaches its maximum 
at $y_0 \in \mathbb R_+$.  We next study  two different cases:\\
Case 1:  $y_0\in (0, 1)$. Then 
\bea
\|F\|_{L^\infty(\{0\} \times \R_+)} \leq 2\sqrt{\pi T}\exp( T+2\norm{\w s}_{L^\infty} )\|I_1\|_{L^\infty(\w s(t), +\infty)}\leq C \|I_1\|_{L^\infty(\w s(t), +\infty)}, 
\eea
which combined with inequality \eqref{bbb}, and  Lemma \ref{lll}, yield 
\bean \label{ffff1}
\norm{u_0-\w u_0}_{L^2(0,b)} \leq \phi(C\norm{s-\w s}_{W^{1,\infty}}).
\eean
Case 2:  $y_0\geq 1$. We have 
\bean \label{sss}
\|F\|_{L^\infty(\{0\} \times \R_+)}= \abs{F(y_0)} \leq 2\sqrt{\pi T}e^{(T+2\norm{\w s}_{L^\infty} )y_0^2}
\|I_1\|_{L^\infty(\w s(t), +\infty)}.
\eean
On the other hand, since $f(0)=f(2b)=0$, a simple integration  by parts 
gives
\bea
F(y) = \frac{1}{y}\int_0^{2b}e^{-yr}f^\prime(r) dr, 
\eea
for $y\geq 1$, which implies 
\bea
\abs{F(y)}\leq \frac{C}{y}.
\eea
Hence 
\bea
\|F\|_{L^\infty(\{0\} \times \R_+)} \leq \frac{C}{y_0},
\eea
or equivalently 
\bean \label{zzz}
y_0 \leq C\|F\|_{L^\infty(\{0\} \times \R_+)}^{-1}.
\eean
Combining  inequalities \eqref{sss} and \eqref{zzz}, we obtain
\bea
 \exp( \frac{-C}{\|F\|_{L^\infty(\{0\} \times \R_+)}^2})
\|F\|_{L^\infty(\{0\} \times \R_+)} \leq 2\sqrt{\pi T}\|I_1\|_{L^\infty(\w s(t), +\infty)}.
\eea
Simple calculations show that 
\bea
\exp( \frac{-(C+e^{-1})}{t^2}) \leq \exp( \frac{-C}{t^2}) t, \quad \textrm{  for all  }  t \in (0, 1). 
\eea
Consequently 
\bea
 \exp( \frac{-C}{\|F\|_{L^\infty(\{0\} \times \R_+)}^2}) \leq 2\sqrt{\pi T}\|I_1\|_{L^\infty(\w s(t), +\infty)},
\eea
or equivalently 
\begin{equation}
 \|F\|_{L^\infty(\{0\} \times \R_+)}  \leq \frac{C}{\left|\ln\left(\|I_1\|_{L^\infty(\w s(t), +\infty)}\right)
 \right|^{\frac{1}{2}}}.
\end{equation}
which combined with inequality \eqref{bbb}, and  Lemma \ref{lll}, provide 
\bean \label{ffff2}
\norm{u_0-\w u_0}_{L^2(0,b)} \leq 
 \phi(C\left|\ln\left(C\norm{s-\w s}_{W^{1,\infty}} \right) \right|^{-\frac{1}{2}}).
\eean

We deduce from estimates \eqref{ffff1} and \eqref{ffff2} the following bound 
\bean \label{ffff3}
\norm{u_0-\w u_0}_{L^2(0,b)} \leq 
 \phi(\psi(C\norm{s-\w s}_{W^{1,\infty}})),
\eean

where 
\bean \label{psi}
\psi(r)  = C\max(r, \frac{1}{\abs{\ln(r)}^{\frac{1}{2}} }), \quad \textrm{for  } r>0. 
\eean

By studying the variation of the 
function $r \rightarrow\phi(\psi(\left|\ln(Cr) \right|^{-\frac{1}{2}}))$ on $\mathbb R_+,$
one can easily show the existence of a constant $\varepsilon_0>0$ small enough 
that only depends  on  $\widehat{M},b, T,\hat{H}, u_0$, such that the estimate
\bean \label{www}
\norm{u_0-\w u_0}_{L^2(0,b)} \leq 
\frac{C}{\left| \ln( \left| \ln \left(\norm{s-\w s}_{W^{1,\infty}} \right) \right| )  \right|^{\frac{1}{4}} },
\eean
holds for $\norm{s-\w s}_{W^{1,\infty}}  < \varepsilon_0$,
which ends the proof of the main theorem.
\end{proof}

\section{Numerical Analysis}
\setcounter{equation}{0}

\subsection{Numerical Approximation} 
Since $u(s(t),t)=0$, we deduce from  \eqref{u}, the following integral equation
\begin{equation}\label{ust}
 \int^b_0 N(s(t),\xi;t,0)u_0(\xi) d\xi =\intt N(s(t),s(\t);t,\t) \dot s(\t) d\t -  \intt N(s(t),0;t,\t)h(\t) d\t.
\end{equation} 
We consider a uniform grid of the temporal interval $[0,T]$, with a time step $\Delta t= \frac{T}{N}$ discretization of $t$ where $t_j=j\Delta t, \  j=0,...,N$, and a uniform grid of spatial interval $[0,b]$ with a space step $\Delta \xi =\frac{b}{M}$ discretization of $\xi$ where $\xi_i= i \Delta \xi, \  i=0,...,M$. Then the equation (\ref{ust}) becomes
\begin{eqnarray*}
 \sum_{k=0}^{M-1}\int^{\xi_{k+1}}_{\xi_k} N(s(t_i),\xi;t_i,0)u_0(\xi) 
 d\xi \\= \sum_{j=0}^{N-1} \int_{t_j}^{t_{j+1}}
 N(s(t_i),s(\t);t_i,\t) \dot s(\t) d\t - \sum_{j=0}^{N-1}\int_{t_j}^{t_{j+1}} 
 N(s(t_i),0;t_i,\t)h(\t) d\t. 
\end{eqnarray*} 
For example if we use a quadrature formula on one point, we get for $i=1,...,N$ 
\begin{align}\label{sum_of_N}
\sum_{k=0}^{M-1} N(s(t_i),\xi_k;t_i,0)u_0(\xi_k) \Delta \xi = \sum_{j=0}^{N-1}  N(s(t_i),s(\t_j);t_i,\t_j) \dot s(\t_j) \Delta t -\sum_{j=0}^{N-1} N(s(t_i),0;t_i,\t_j)h(\t_j) \Delta t , 
\end{align} 
where $\t_j \in [t_j,t_{j+1}]$ for $j=1,...,N-1$ and $\xi_k \in [ \xi_k,\xi_{k+1}]$ for $k=1,...,M-1$. \\

This system can be represented by 
\begin{equation}\label{Ag}
\m A U_0=g, 
\end{equation}
where $\m A$ denotes a matrix depending on the quadrature formulas,
$U_0:=(U_{0,k})=(u_0(\xi_k))$ denotes the vector of the unknown initial condition of the inverse problem, and $g:=(g_i)$ is the vector representing the right hand side of the equation (\ref{sum_of_N}). Next, we shall use Gauss-Legendre formula for the numerical integration. \\

In this work, we will consider two types of regularization of the inverse problem. The first method is a Tikhonov iterative regularization which consists on computing  an approximate solution defined by \cite{engl1996regularization, bruckner2000tikhonov}
\begin{align}
(\m A^{tr} \m A +\lambda I)U_{0,m+1}&=\m A ^{tr} g+\lambda U_{0,m}, \label{Tikhonov_equation}\\
U_{0,0}&=0,
\end{align}
where the superscript $^{tr}$ denotes the transpose of a matrix, $I$ the identity matrix, and $\lambda > 0$ is the regularization parameter. 

The second method is a Landweber method where the approximate solution is defined by the sequence \cite{landweber1951iteration, hanke1995convergence}
\begin{align}
U_{0,m+1}&=(I-\lambda \m A^{tr} \m A)U_{0,m} +\lambda \m A^{tr} g, \label{landweber_equation} \\
U_{0,0}&=0,
\end{align}
where $\lambda\leq \dfrac{1}{\norm{\m A}^2}$ is a positive constant.

\subsection{Numerical results}
\subsection{Example 1}
In this example \cite{knabner1985control} the moving boundary is given by the nonlinear function 
\begin{equation}
s(t)=\sqrt{t+\frac{1}{4}},  \ t\in[0,1].
\end{equation}
and has the Neumann boundary condition
\begin{equation}
h(t)=  \ds\dfrac{exp(\frac{1}{4})}{2\sqrt{t+\frac{1}{4}}},
\end{equation} 
where $erf(x)$ is the error function given by 
$erf(x)= \frac{2}{\pi} \ds\int^x_0 exp(-t^2) \ dt$.
This example has the initial condition
$u_0(x)=\dfrac{exp(\frac{1}{4}) \sqrt{\pi}} {2}   \big( erf(\frac{1}{2}) - erf(x)\big), \ x\in [0,\frac{1}{2}]$.

\begin{figure}[H]
\centering
\includegraphics[width=12.5cm,height=7.5cm]{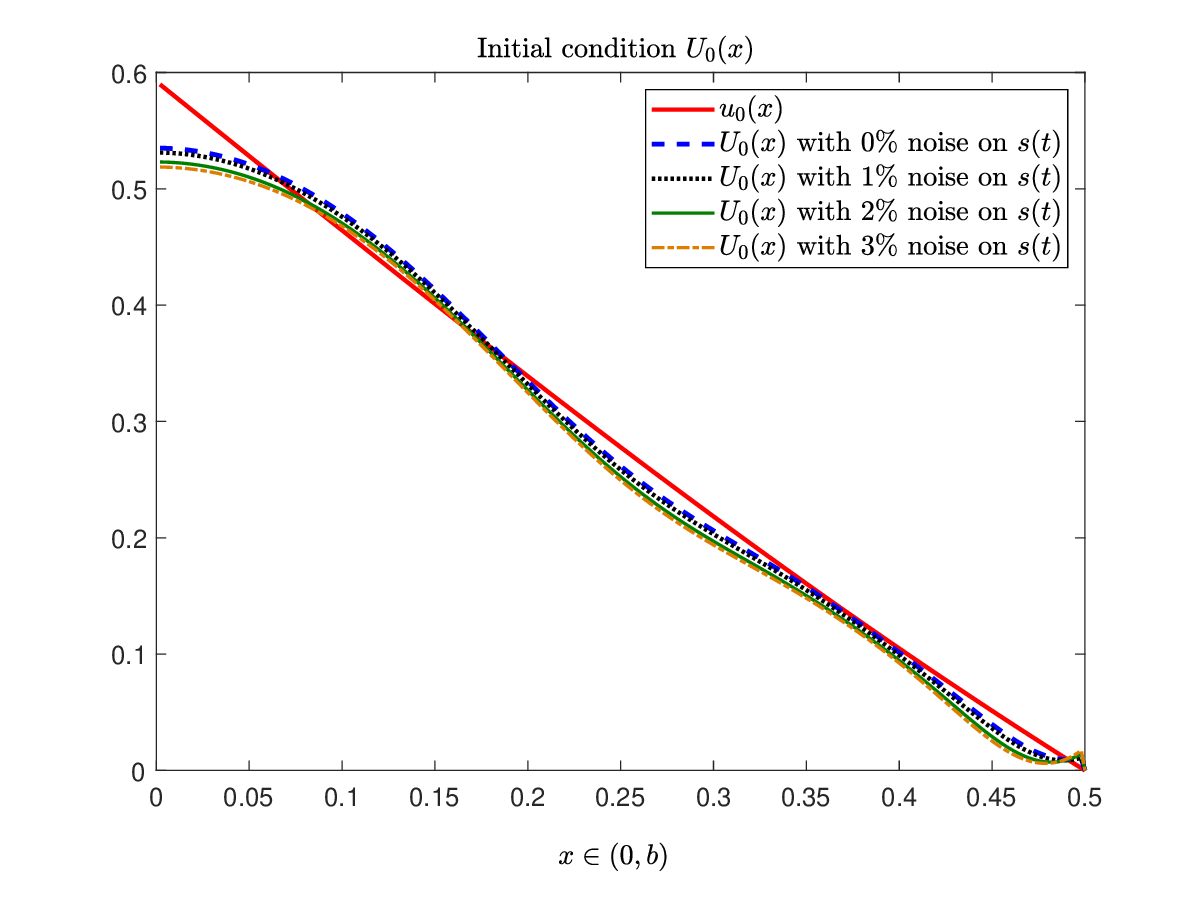}
\caption{The exact initial condition $u_0(x)$ and approximate solution with different Gaussian noise levels obtained with $\lambda=10^{-3}$, $M=250$ using Tikhonov Regularization method.}
\label{U0KnabnerTikhonov}
\end{figure}
\begin{table}[H]
\centering
\begin{tabular}{|c|c|c|}
\hline
 $\lambda$ & Noise on $s(t)$ ($\%$) &$\frac{\norm{u_0-U_0}_2}{\norm{u_0}_2}$  \\
\hline
  $10^{-3}$& 0 $\%$ & 0.0425 \\
\hline
 $10^{-3}$ & 1 $\%$ & 0.0472 \\
 \hline
 $10^{-3}$ & 2 $\%$ & 0.0571 \\
\hline
 $10^{-3}$ & 3 $\%$ & 0.0669 \\
\hline
\end{tabular}
\caption{Relative errors using Tikhonov method.}
\label{Table_errors_Knanber_Tikhonov}
\end{table}
In  Figure \ref{U0KnabnerTikhonov}, we present the exact initial condition and its reconstruction using the Tikhonov Regularization method. 
In order to test the stability of our inverse problem, we add a different level of gaussian noise to the data $s(t)$: $1$, $2$ and $3\%$. 
The Table \ref{Table_errors_Knanber_Tikhonov} compiles the relative errors of reconstruction of the function describing the initial condition using the Tikhonov Regularization method. It shows that the accuracy is altered with noise and the relative error on the numerical solution is lower than $0.07$. 

\begin{figure}[H]
\centering
\includegraphics[width=11.5cm,height=7.5cm]{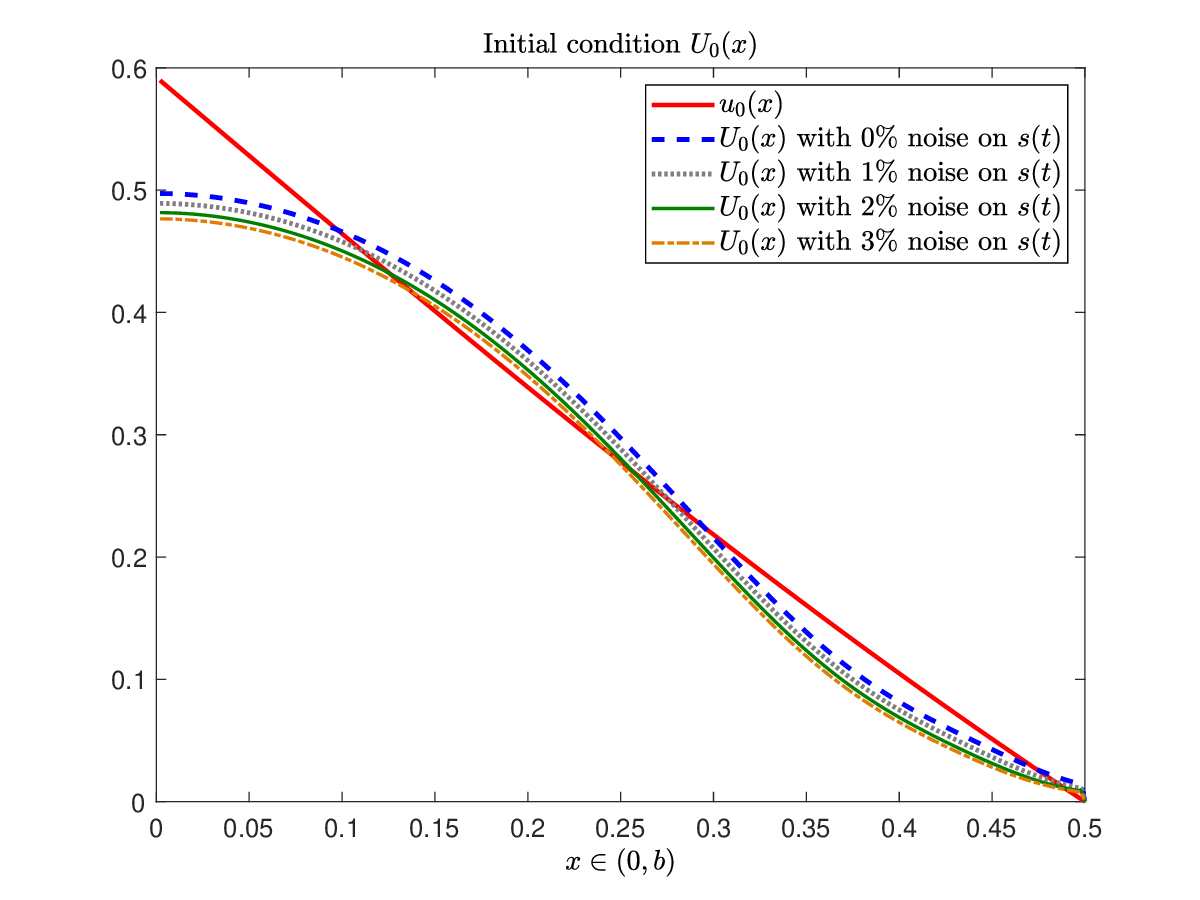}
\caption{The exact initial condition $u_0(x)$ and approximate solution with different Gaussian noise levels obtained with $M=250$ using Landweber method.}
\label{U0Knabner_Landweber}
\end{figure}
\begin{table}[H]
\centering
\begin{tabular}{|c|c|}
\hline
  Noise on $s(t)$ ($\%$) &$\frac{\norm{u_0-U_0}_2}{\norm{u_0}_2}$  \\
\hline
   0 $\%$ & 0.0846 \\
\hline
   1 $\%$ & 0.0917 \\
 \hline
 2 $\%$ & 0.1026\\
\hline
 3 $\%$ & 0.1115 \\
\hline
\end{tabular}
\caption{Relative errors using Landweber method.}
\label{Table_errors_Knanber_Landweber}
\end{table}
Considering the same example but this time with the Landweber method, the results in the Table \ref{Table_errors_Knanber_Landweber} show that the relative error of the initial condition $U_0(x)$ varies from $8.5 \times 10^{-2}$ to $1.1 \times 10^{-1}$ by adding $1\%$, $2\%$ and $3\%$ of Gaussian noise on the data $s(t)$. 
We particularly notice from Tables \ref{Table_errors_Knanber_Tikhonov} and \ref{Table_errors_Knanber_Landweber} that the accuracy of the solution $U_0(x)$ with Tikhonov Regularization method is higher compared to the solution obtained with Landweber method.

\subsection{Example 2}
This example \cite{johansson2011method} has a moving boundary given by the linear function
\begin{equation}
s(t)=\sqrt{2}-1+\frac{t}{\sqrt{2}}, \ t\in[0,1].
\end{equation}
We take the exact solution given by 
\begin{equation}
u(x,t)=-1+\exp \big( 1-\frac{1}{\sqrt{2}} +\frac{t}{2} -\frac{x}{\sqrt{2}} \big) , \ \ [x,t]\in [0,s(t)]\times [0,1].
\end{equation}
Therefore, this example has the following initial and boundary conditions
\begin{align}
b&=s(0)=\sqrt{2}-1, \\
-u_x(0,t)&=\frac{1}{\sqrt{2}} \exp\big ( 1-\frac{1}{\sqrt{2}} +\frac{t}{2} \big), \ t\in[0,1]\\
u(s(t),t)&=0, \  t\in(0,1], \\
u_x(s(t),t)&= \dot s(t)=-\frac{1}{\sqrt{2}}, \ t\in(0,1].
\end{align}
In this example we wish to recover the initial condition at $t = 0$ given by 
\begin{equation}
u_0(x)= -1+\exp \big( 1-\frac{1}{\sqrt{2}} -\frac{x}{\sqrt{2}} \big), x\in[0,b].
\end{equation}

\begin{figure}[H]
\centering
\includegraphics[width=12.5cm,height=7.5cm]{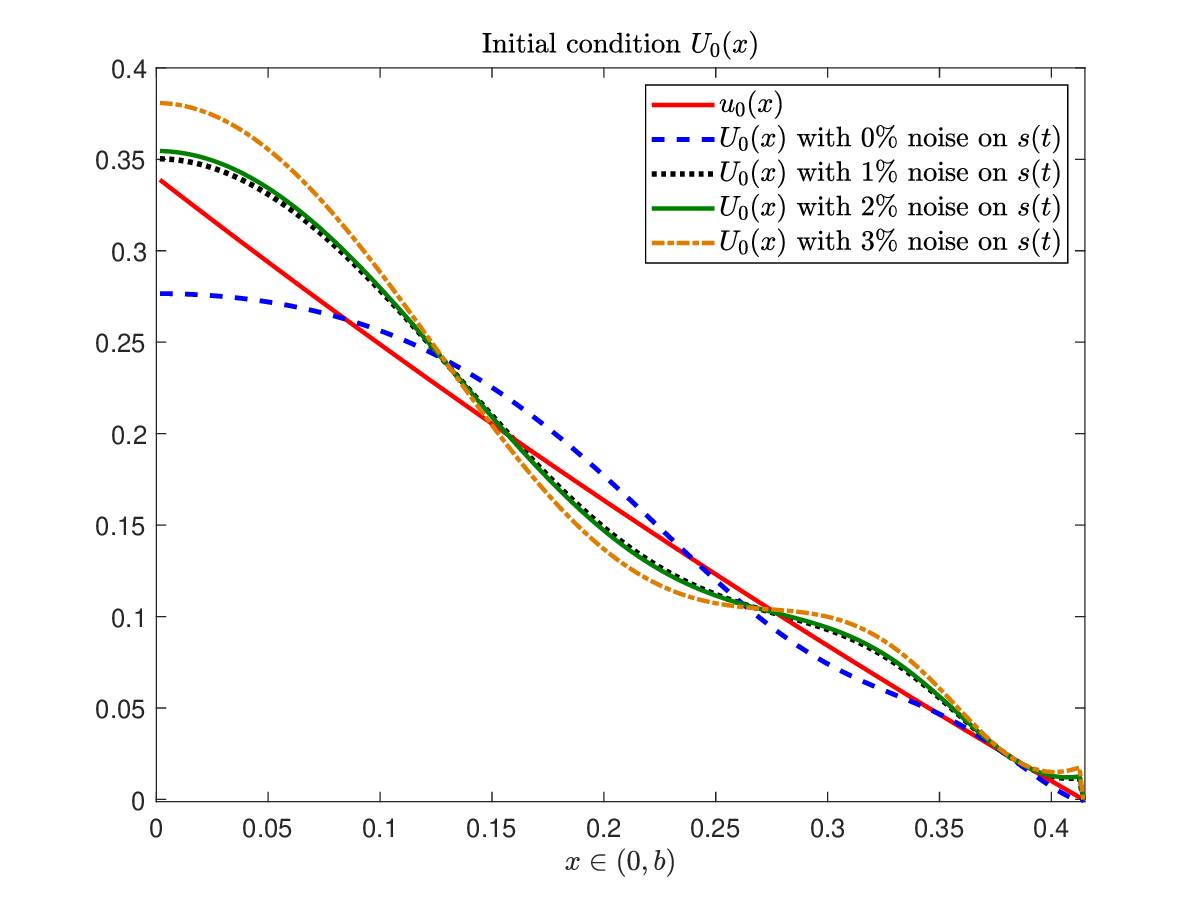}
\caption{The exact initial condition $u_0(x)$ and the approximate solution with different Gaussian noise levels obtained with $\lambda=10^{-2}$, $M=250$ using Tikhonov method.}
\label{U0_Lesnic_Tikhonov}
\end{figure}
\begin{table}[H]
\centering
\begin{tabular}{|c|c|c|}
\hline
 $\lambda$ & Noise on $s(t)$ ($\%$) &$\frac{\norm{u_0-U_0}_2}{\norm{u_0}_2}$  \\
\hline
  $10^{-2}$& 0 $\%$ & 0.0953 \\
\hline
 $10^{-2}$ & 1 $\%$ & 0.0997 \\
 \hline
 $10^{-2}$ & 2 $\%$ & 0.1082 \\
\hline
 $10^{-2}$ & 3 $\%$ & 0.1465 \\
\hline
\end{tabular}
\caption{Relative errors using Tikhonov method.}
\label{Table_errors_Lesnic_Tikhonov}
\end{table}

\begin{figure}[H]
\centering
\includegraphics[width=11.5cm,height=7.5cm]{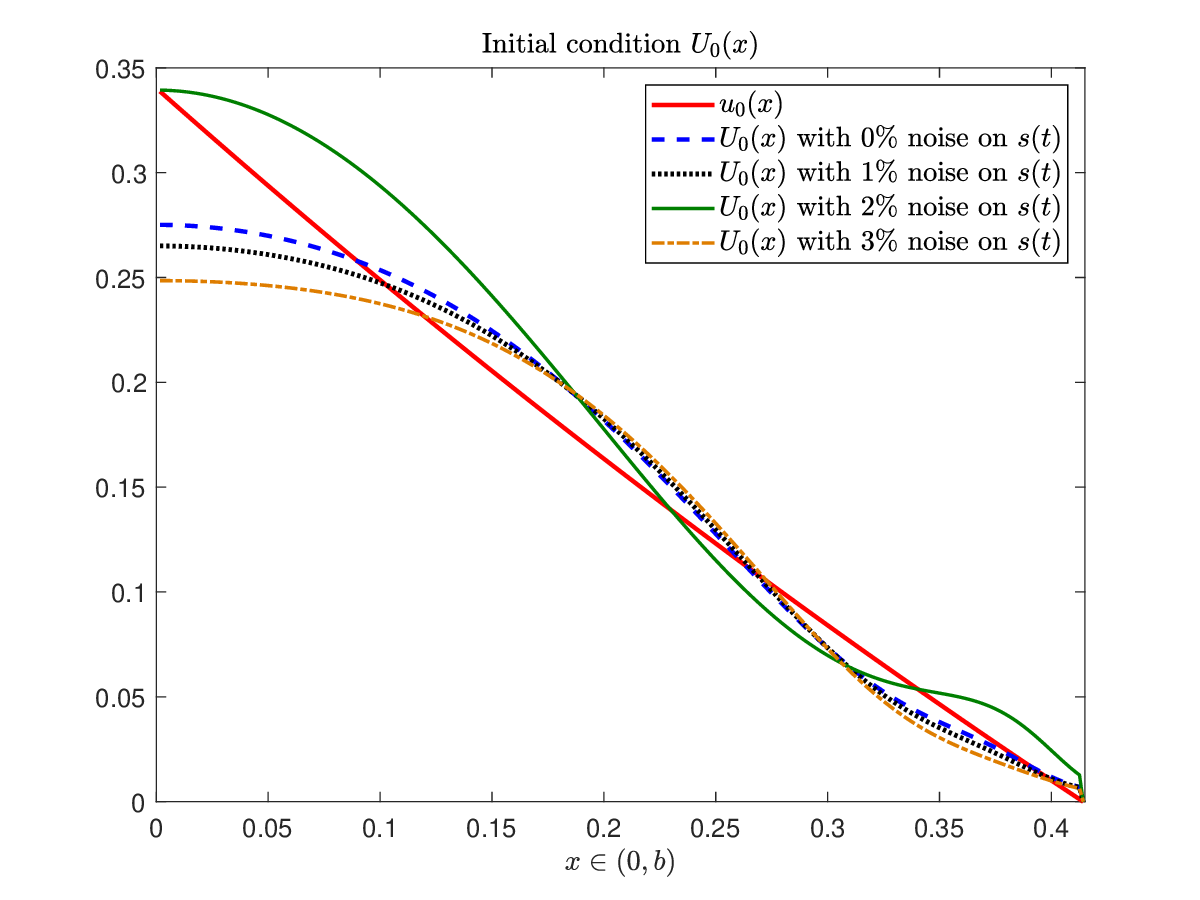}
\caption{The exact initial condition $u_0(x)$ and approximate solution with different Gaussian noise levels obtained with $M=250$ using Landweber method.}
\label{U0_Lesnic_Landweber}
\end{figure}

\begin{table}[H]
\centering
\begin{tabular}{|c|c|}
\hline
  Noise on $s(t)$ ($\%$) &$\frac{\norm{u_0-U_0}_2}{\norm{u_0}_2}$  \\
\hline
   0 $\%$ & 0.1017 \\
\hline
   1 $\%$ & 0.1188 \\
 \hline
 2 $\%$ & 0.1321 \\
\hline
 3 $\%$ & 0.1520 \\
\hline
\end{tabular}
\caption{Relative errors using Landweber method.}
\label{Table_errors_Lesnic_Landweber}
\end{table}
In Figures \ref{U0_Lesnic_Tikhonov} and \ref{U0_Lesnic_Landweber}, we plotted  the exact initial condition and the approximate solution using respectively the Tikhonov and Landweber Regularization methods  with  different added Gaussian noise levels. 
The Tables \ref{Table_errors_Lesnic_Tikhonov} and \ref{Table_errors_Lesnic_Landweber} provide the relative errors of the solution $U_0$ using respectively the Tikhonov and Landweber Regularization methods. We remark that the accuracy of the initial condition using Tikhonov Regularization is higher compared to the solution with Landweber method.

{
\subsection{Example 3}
In this numerical test, we consider an example where the analytic solution is not available. We resolve the direct Stefan problem \eqref{1.16}-\eqref{1.21} with the initial and boundary conditions cited below (\ref{conditions}) by using the Boundary Immobilisation Method (BIM)  \cite{ghanmi2020identification} in order to calculate the moving boundary $s(t)$. Then we intend to recover the initial temperature $U_0(x)$. 

\begin{equation}
u_0(x)=(3-x)\sqrt{\abs{3-2x}} , \ x\in [0,3], \quad \quad
h(t)=\sqrt{t+1}, \ t\in[0,3].
\label{conditions}
\end{equation}

\begin{figure}[H]
\centering
\includegraphics[width=11.5cm,height=7.5cm]{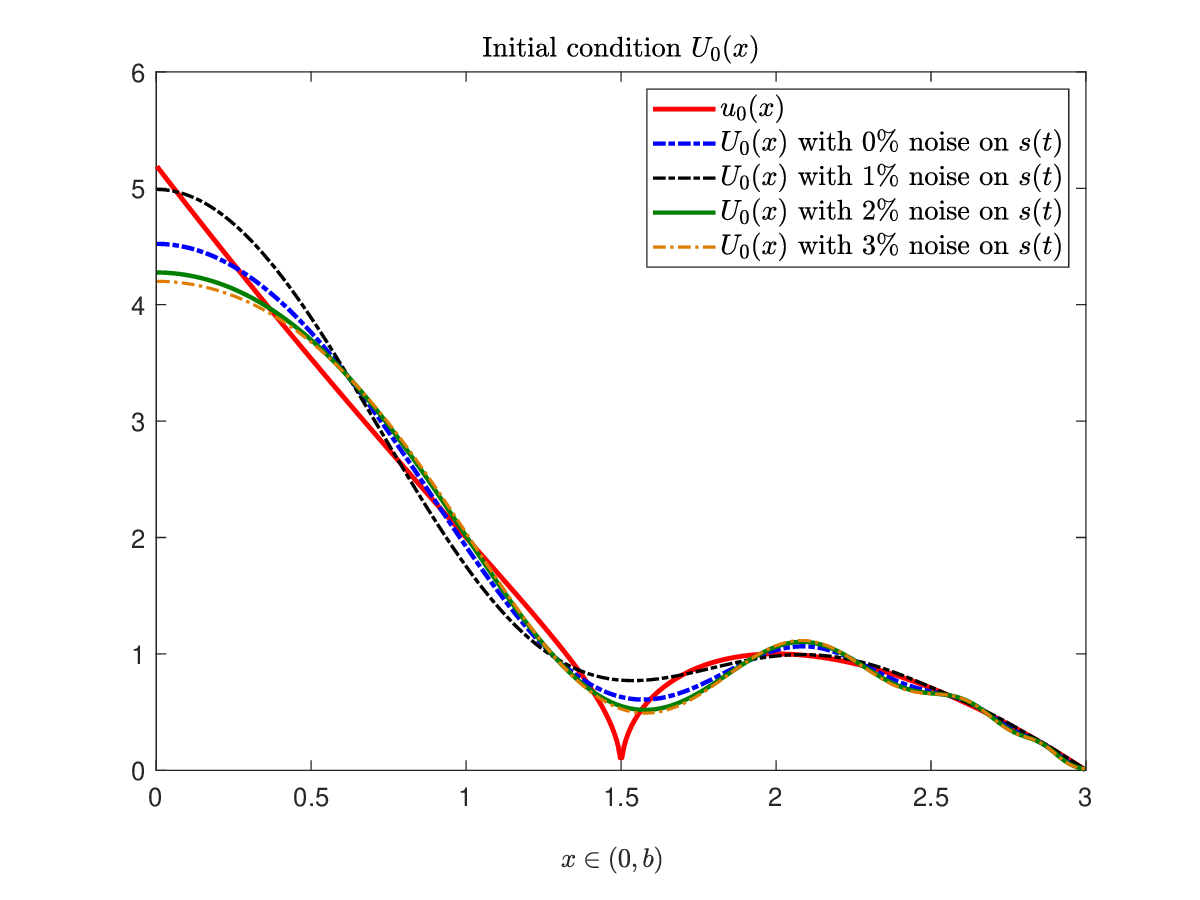}
\caption{The initial condition $u_0(x)$ and approximate solution with different Gaussian noise levels obtained with $\lambda=10^{-3}$ and $M=250$ using Tikhonov method.}
\label{U0_ex3_Tikhonov}
\end{figure}
In the Figure \ref{U0_ex3_Tikhonov}, we present the exact initial condition and the approximate solutions with different added Gaussian noise levels on the data $s(t)$ using the Tikhonov method. From the Table \ref{Table_errors_ex3_Tikhonov}, we show that the relative error of the initial condition $U_0(x)$ using the Tikhonov Regularization varies from $7.14 \times 10^{-2}$ to $1.00 \times 10^{-1}$ by adding different levels of Gaussian noise on the data $s(t)$.

\begin{table}[H]
\centering
\begin{tabular}{|c|c|c|}
\hline
 $\lambda$ & Noise on $s(t)$ ($\%$) &$\frac{\norm{u_0-U_0}_2}{\norm{u_0}_2}$  \\
\hline
  $10^{-3}$& 0 $\%$ & 0.0714\\
\hline
 $10^{-3}$ & 1 $\%$ & 0.0866 \\
 \hline
 $10^{-3}$ & 2 $\%$ & 0.0916 \\
\hline
 $10^{-3}$ & 3 $\%$ & 0.1002 \\
\hline
\end{tabular}
\caption{Relative errors using Tikhonov method.}
\label{Table_errors_ex3_Tikhonov}
\end{table}
We take the same example again but this time with the Landweber method, we notice from the table \ref{Table_errors_ex3_Landweber} that the relative error of the initial condition $U_0(x)$ varies from $6.90 \times 10^{-2}$ to $1.13 \times 10^{-1}$ by adding $1\%$, $2\%$ and $3\%$ of Gaussian noise on the data $s(t)$. The exact and approximate solutions for different noise levels using Landweber method are illustrated in Figure \ref{U0_ex3_Landweber}. 

\begin{figure}[H]
\centering
\includegraphics[width=11.5cm,height=7.5cm]{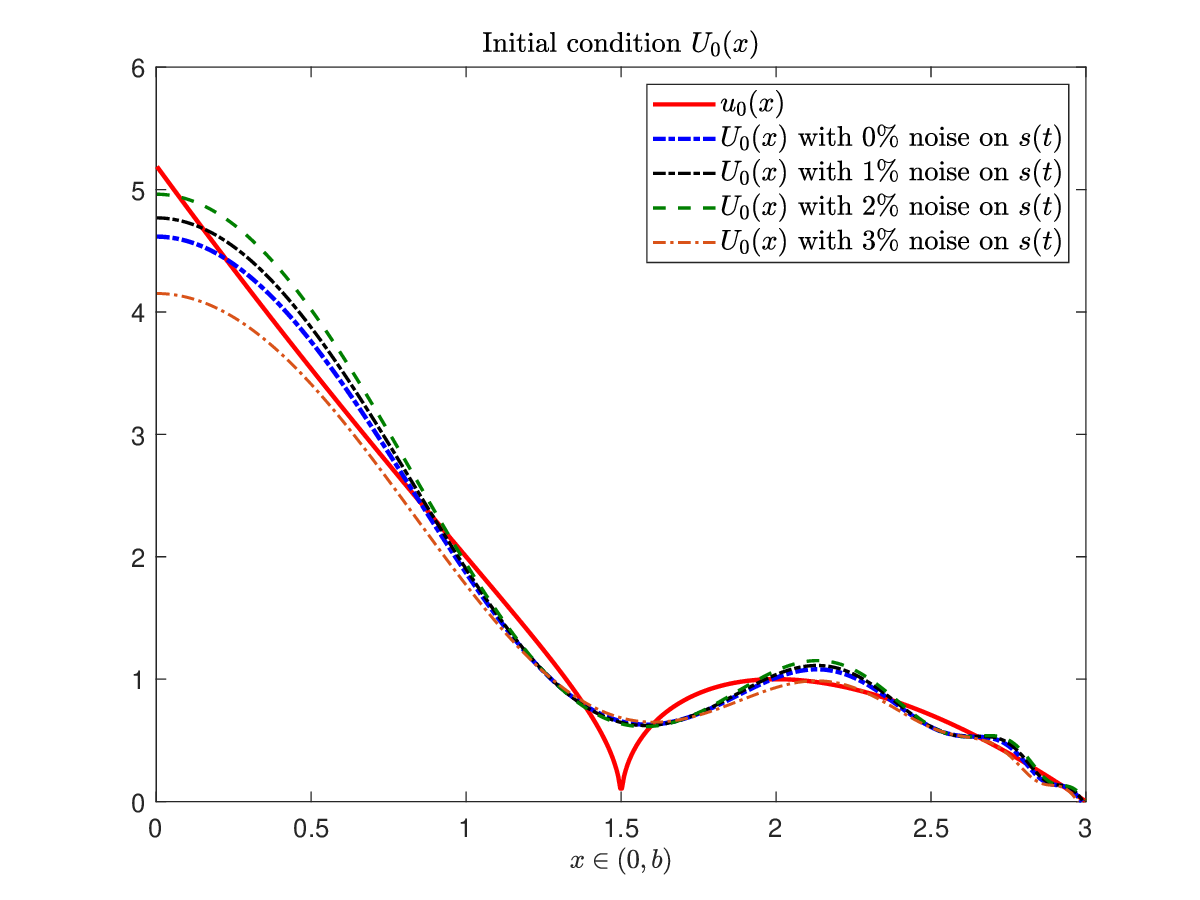}
\caption{The initial condition $u_0(x)$ and approximate solution with different Gaussian noise levels obtained with $M=250$ using Landweber method.}
\label{U0_ex3_Landweber}
\end{figure}

\begin{table}[H]
\centering
\begin{tabular}{|c|c|}
\hline
  Noise on $s(t)$ ($\%$) &$\frac{\norm{u_0-U_0}_2}{\norm{u_0}_2}$  \\
\hline
   0 $\%$ & 0.0690 \\
\hline
   1 $\%$ & 0.0755 \\
 \hline
 2 $\%$ & 0.0970 \\
\hline
 3 $\%$ & 0.1132 \\
\hline
\end{tabular}
\caption{Relative errors using Landweber method.}
\label{Table_errors_ex3_Landweber}
\end{table}
}

\begin{remark}
Notice that in the two first considered examples the explicit initial conditions  are analytic functions of the spatial 
variable. This explains their relatively good recovery far away from the endpoints, 
from the measurement of the free boundary. In Table \ref{Table_errors_ex3_Landweber}, we observe 
that the reconstruction is very sensitive to the noise level. This is in accordance with the 
obtained stability estimate, and the  ill-posedness of the inverse problem. 
 \end{remark}

\bibliographystyle{plain}
\bibliography{references}

\end{document}